\documentclass[nospthms,smallextended]{svjour3}        
 
\usepackage{mathptmx} 
\usepackage{amssymb,amsmath,amsthm,newlfont,enumerate,graphicx}

\theoremstyle{plain}
\newtheorem{theorem}{Theorem}[section]
\newtheorem{lemma}[theorem]{Lemma}
\newtheorem{corollary}[theorem]{Corollary}
\newtheorem{proposition}[theorem]{Proposition}

\theoremstyle{remark}
\newtheorem*{remark}{Remark}

\newcommand{\CC}{{\mathbb C}}
\newcommand{\RR}{{\mathbb R}}
\newcommand{\quarter}{\textstyle\frac{1}{4}}
\DeclareMathOperator{\inter}{int}

\journalname{Minda Conference Proceedings}

\begin{document}

\title{Cross-sections of multibrot sets
\thanks{TR supported by grants from NSERC and the Canada research chairs program}}

\author{Line Baribeau  \and Thomas Ransford}

\institute{Line Baribeau \at
              D\'epartement de math\'ematiques et de statistique, Universit\'e Laval,\\
              1045 avenue de la M\'edecine, Qu\'ebec (QC), Canada G1V 0A6\\
              \email{line.baribeau@mat.ulaval.ca}           
           \and
           Thomas Ransford \at
              D\'epartement de math\'ematiques et de statistique, Universit\'e Laval,\\
              1045 avenue de la M\'edecine, Qu\'ebec (QC), Canada G1V 0A6\\ 
              Tel.: +14186562131 ext 2738\\
              Fax: +14186565902\\
              \email{thomas.ransford@mat.ulaval.ca}           
}

\date{Received: date / Accepted: date}

\dedication{Dedicated to David Minda on the occasion of his retirement}

\maketitle

\begin{abstract}
We identify the intersection of the multibrot set of $z^d+c$ 
with the rays $\RR^+\omega$, where $\omega^{d-1}=\pm1$.
\end{abstract}

\section{Introduction}

Let $d$ be an integer with $d\ge2$.
Given $c\in\CC$, we define
\[
p_c(z):=z^d+c
\quad\text{and}\quad
p_c^{[n]}:=p_c\circ\dots\circ p_c\quad\text{($n$ times)}.
\]
The corresponding generalized Mandelbrot set, 
or \emph{multibrot set}, is defined by
\[
M_d:=\Bigl\{c\in\CC: \sup_{n\ge0}|p_c^{[n]}(0)|<\infty\Bigr\}.
\]
Of course $M_2$ is just the classical Mandelbrot set. 
Computer-generated images of $M_3$ and $M_4$ are pictured in Figure~\ref{F:M3M4}.
Multibrot sets have been extensively studied in the literature.
Schleicher's article \cite{Sch04} 
contains a wealth of background material on them.

\begin{figure*}[htb]
\begin{minipage}{0.48\linewidth}
\centering
\includegraphics[scale=0.378, trim=0 7 0 0,clip=true]{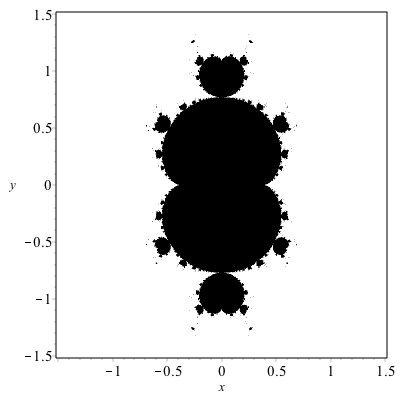}
\end{minipage}
\begin{minipage}{0.48\linewidth}
\centering
\includegraphics[scale=0.4]{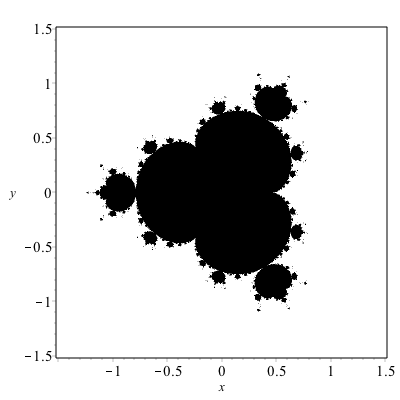}
\end{minipage}
\caption{The multibrot sets $M_3$ and $M_4$}\label{F:M3M4}
\end{figure*}

We mention here some elementary properties of multibrot sets.
First of all, they exhibit $(d-1)$-fold rotational invariance, namely   
\begin{equation}\label{E:rot}
M_d=\omega M_d \qquad(\omega\in\CC,~\omega^{d-1}=1).
\end{equation}
Indeed, for these $\omega$,  writing $\phi(z):=\omega z$, 
we have $\phi^{-1}\circ p_c\circ\phi=p_{c/\omega}$, 
so $p_c^{[n]}(0)$ remains bounded 
if and only if $p_{c/\omega}^{[n]}(0)$ does.
(In fact, the rotations in \eqref{E:rot} are the only rotational symmetries of $M_d$. 
The paper of Lau and Schleicher \cite{LS96} contains an elementary proof of this fact.)

Also, writing $\overline{D}(0,r)$ for 
the closed disk with center $0$ and radius $r$, 
we have the inclusions
\[
\overline{D}(0,\alpha(d))\subset M_d\subset \overline{D}(0,\beta(d)),
\]
where 
\[
\alpha(d):=(d-1)d^{-d/(d-1)}
\quad\text{and}\quad
\beta(d):=2^{1/(d-1)}.
\]
The first inclusion follows from the fact that, if $|c|\le\alpha(d)$, 
then the closed disk $\overline{D}(0,d^{-1/(d-1)})$ is mapped into itself by $p_c$, 
and consequently the sequence $p_c^{[n]}(0)$ is bounded. 
For the second inclusion, we observe that, 
if $|c|>\beta(d)$, then by induction
$|p_c^{[n+2]}(0)|\ge (2d)^n(|c|^d-2|c|)$ for all $n\ge0$,
and the right-hand side of this inequality tends to infinity with $n$.

When $d$ is odd, we have
\begin{equation}\label{E:odd}
M_d\cap\RR=[-\alpha(d),\alpha(d)].
\end{equation}
This equality was conjectured by Paris\'e and Rochon in \cite{PR15a},
and proved by them in \cite{PR15b}. 
Also, when $d$ is even, we have
\begin{equation}\label{E:even}
M_d\cap\RR=[-\beta(d),\alpha(d)].
\end{equation}
This equality was also conjectured in \cite{PR15a}, 
and subsequently proved in \cite{PRR16}.
When $d=2$, it reduces to the well-known equality 
$M_2\cap\RR=[-2,\quarter]$.

By virtue of the rotation-invariance property \eqref{E:rot}, 
the equalities \eqref{E:odd} and \eqref{E:even} yield information about 
the intersection of $M_d$ with certain rays emanating from zero. 
Indeed, if $\omega^{d-1}=1$, then 
\[
M_d\cap\RR^+\omega=\{t\omega:0\le t\le \alpha(d)\},
\]
and if $\omega^{d-1}=-1$ and $d$ is even, then 
\[
M_d\cap\RR^+\omega=\{t\omega:0\le t\le \beta(d)\}.
\]
This leaves open the case when $\omega^{d-1}=-1$ and $d$ is odd.
The purpose of this note is to fill the gap.
The following theorem is our main result.

\begin{theorem}\label{T:gamma}
If $\omega^{d-1}=-1$ and $d$ is odd, then 
\[
M_d\cap\RR^+\omega=\{t\omega:0\le t\le \gamma(d)\},
\]
where
\begin{equation}\label{E:gamma}
\gamma(d):=d^{-d/(d-1)}\bigl(\sinh(d\xi_d)+d\sinh(\xi_d)\bigr),
\end{equation}
and
$\xi_d$ is the unique positive root of the equation
$\cosh(d\xi_d)=d\cosh(\xi_d)$.
\end{theorem}

When $d=3$, one can use the relation 
$\cosh(3x)=4\cosh^3x-3\cosh x$ to derive the exact formula 
$\gamma(3)=\sqrt{32/27}$, which yields

\begin{corollary}
$M_3\cap i\RR=\{iy: |y|\le \sqrt{32/27}\}.$
\end{corollary}

In comparison, 
note that \eqref{E:odd} gives $M_3\cap\RR=\{x:|x|\le 2/\sqrt{27}\}$. 
See  Figure~\ref{F:M3M4}.

The first few values of $\alpha(d), \beta(d),\gamma(d)$ are tabulated in Table~\ref{Tb:abc} for comparison.

\begin{table}[htb]
\caption{Values of $\alpha(d),\beta(d),\gamma(d)$ for $2\le d\le 12$}
\label{Tb:abc}
\renewcommand{\arraystretch}{1.1}
\begin{tabular}{|r|c|c|c|}
\hline
$d$ &$\alpha(d)$ &$\beta(d)$  &$\gamma(d)$\\ 
\hline
$2$ &$0.250000000$ &$2.000000000$ &$1.100917369$\\
$3$ &$0.384900179$ &$1.414213562$ &$1.088662108$\\
$4$ &$0.472470394$ &$1.259921050$ &$1.078336651$\\
$5$ &$0.534992244$ &$1.189207115$ &$1.069984489$\\
$6$ &$0.582355932$ &$1.148698355$ &$1.063192242$\\
$7$ &$0.619731451$ &$1.122462048$ &$1.057591279$\\ 
$8$ &$0.650122502$ &$1.104089514$ &$1.052904317$\\
$9$ &$0.675409498$ &$1.090507733$ &$1.048928539$\\
$10$ &$0.696837314$ &$1.080059739$ &$1.045514971$\\ 
$11$ &$0.715266766$ &$1.071773463$ &$1.042552690$\\
$12$ &$0.731314279$ &$1.065041089$ &$1.039957793$\\
\hline
\end{tabular}
\renewcommand{\arraystretch}{1.0}
\end{table}

It can be shown that $\gamma(d)>1$ for all $d$, and that
\[
\gamma(d)=2^{1/d+O((\log d)^2/d^2)) }
\quad\text{as~}d\to\infty.
\]
These statements will be justified later.

\section{Proof of Theorem~\ref{T:gamma}}

In this section we suppose that $d$ is an odd integer with $d\ge3$. 
If $\omega^{d-1}=-1$, then, writing $\phi(z):=\omega z$, 
we have $\phi^{-1}\circ p_c\circ \phi=q_{c/\omega}$, where
\[
q_c(z):=-z^d+c.
\]
Thus $M_d\cap\RR^+\omega=\omega (N_d\cap \RR^+)$, where
\[
N_d:=\Bigl\{c\in\CC: \sup_{n\ge0}|q_c^{[n]}(0)|<\infty\Bigr\}.
\]
We now seek to identify $N_d\cap\RR^+$. 
We shall do this in two stages.

\begin{lemma}\label{T:N}
Let $d$ be an odd integer with $d\ge3$. Then
\[
N_d\cap\RR^+=[0,~\mu(d)],
\]
where
\[
\mu(d):=\max\bigl\{a-b^d:a,b\ge0,~a^d+b^d=a+b\bigr\}.
\]
\end{lemma}

\begin{proof}
Consider first the case $c\in[0,1]$.
In this case we have $q_c(0)=c$ and $q_c(c)=-c^d+c\ge0$.
Since $q_c$ is a decreasing function, 
it follows that $q_c([0,c])\subset[0,c]$, 
and in particular that $q_c^{[n]}(0)$ is bounded. 
Hence $c\in N_d$ for all $c\in[0,1]$.

Consider now the case $c\in[1,\infty)$. 
Then $q_c(0)=c$ and $q_c^{[2]}(0)=-c^d+c\le0$. 
As $q_c$ is a decreasing function,
it follows that $q_c^{[2n]}(0)$ is a decreasing sequence 
and $q_c^{[2n+1]}(0)$ is an increasing sequence. 
If, further, $c\in N_d$, then $q_c^{[n]}(0)$ is bounded, 
and both of these subsequences converge, 
say $q_c^{[2n+1]}(0)\to a$ and $q_c^{[2n]}(0)\to -b$, 
where $a,b\ge0$. 
We then have $q_c(-b)=a$ and $q_c(a)=-b$, 
in other words $b^d+c=a$ and $a^d-c=b$. 
Adding these equations gives $a^d+b^d=a+b$. 
Summarizing what we have proved: 
if $c\in N_d\cap[1,\infty)$, then $c=a-b^d$, 
where $a,b\ge0$ and $a^d+b^d=a+b$. 
Conversely, if $c$ is of this form, 
then $q_c(-b)=a$ and $q_c(a)=-b$, 
so $[-b,a]$ is a $q_c$-invariant interval containing $0$, 
which implies that $q^{[n]}(0)$ remains bounded, 
and hence  $c\in N_d$. 
Combining these remarks, 
we have shown that
\begin{equation}\label{E:interval}
N_d\cap[1,\infty)=\{a-b^d:a,b\ge0,~a^d+b^d=a+b\}\cap[1,\infty).
\end{equation}
The condition  that $a^d+b^d=a+b$ can be re-written as $h(a)=-h(b)$, 
where $h(x):=x^d-x$. 
Viewed this way, it is more or less clear that 
the right-hand side of \eqref{E:interval} 
is a closed interval containing~$1$, 
so $N_d\cap[1,\infty)=[1,\mu(d)]$, 
where $\mu(d)$ is as defined in the statement of the lemma. 

Finally, putting all of this together, we have shown that 
$N_d\cap\RR^+=[0,\mu(d)]$.
\end{proof}

Next we identify $\mu(d)$  more explicitly.

\begin{lemma} 
$\mu(d)=\gamma(d)$.
\end{lemma}

\begin{proof}
We reformulate the maximization problem defining $\mu(d)$. Set 
\begin{align*}
S&:=\{(a,b)\in\RR^2: a,b\ge0\},\\
f(a,b)&:=a-b^d,\\
g(a,b)&:=a^d+b^d-a-b.
\end{align*}
We are seeking to maximize $f$ over $S\cap\{g=0\}$. 
The set $S\cap\{g=0\}$ is compact and $f$ is continuous, 
so the maximum is certainly attained, say at $(a_0,b_0)$. 
Notice also that $\nabla g\ne0$ at every point of 
$S\cap\{g=0\}$.
There are two cases to consider.

Case 1: $(a_0,b_0)\in\partial S$. 
The condition that $g(a_0,b_0)=0$ then implies that 
\[
(a_0,b_0)=(0,0),(0,1) \text{~or~} (1,0).
\]
The corresponding values of $f(a_0,b_0)$ are $0,-1,1$ respectively. 
Clearly we can eliminate the first two points from consideration. 
As for the third, we remark that 
the directional derivative of $f$ at $(1,0)$ along $\{g=0\}$ 
in the direction pointing into $S$ 
is equal to $1/\sqrt{1+(d-1)^2}$, which is strictly positive. 
So $(1,0)$ cannot be a maximum of $f$ either. 

Case 2: $(a_0,b_0)\in\inter(S)$.
In this case, by the standard Lagrange multiplier argument, 
we must have 
$\nabla f(a_0,b_0)=\lambda\nabla g(a_0,b_0)$ for some $\lambda\in\RR$. 
Writing this out explicitly, we get
\begin{align*}
1&=\lambda(da_0^{d-1}-1),\\
-db_0^{d-1}&=\lambda(db_0^{d-1}-1).
\end{align*}
Dividing the  second equation by the first 
and then simplifying, we obtain
\[
a_0b_0=d^{-2/(d-1)}.
\]
Thus $a_0=d^{-1/(d-1)}e^\xi$ and $b_0=d^{-1/(d-1)}e^{-\xi}$ 
for some $\xi\in\RR$. With this notation, 
the constraint $g(a_0,b_0)=0$ translates to $\cosh(d\xi)=d\cosh(\xi)$, 
and the value of $f$ at $(a_0,b_0)$ is
\[
f(a_0,b_0)=a_0-b_0^d
=\frac{a_0-b_0}{2}+\frac{a_0^d-b_0^d}{2}
=d^{-d/(d-1)}\bigl(d\sinh(\xi)+\sinh(d\xi)\bigr).
\]
There are precisely two roots of $\cosh(d\xi)=d\cosh(\xi)$,
one positive and one negative. 
Necessarily the positive root gives rise to the maximum value of $f$, 
thereby showing that $\mu(d)=\gamma(d)$.
\end{proof}

\begin{remark}
Clearly $f(1,0)=1$. 
The treatment of Case~1 above shows that 
$f$ does not attain its maximum over $S\cap\{g=0\}$ at $(1,0)$, 
and so $\mu(d)>1$. 
This shows that $\gamma(d)>1$, 
thereby justifying a statement made in the introduction.
\end{remark}

\begin{proof}[Proof of Theorem~\ref{T:gamma}]
Combining the various results already obtained in this section, 
we have
\[
M_d\cap\RR^+\omega=\omega(N_d\cap\RR^+)=\omega[0,\mu(d)]=\omega[0,\gamma(d)].
\]
This concludes the proof of Theorem~\ref{T:gamma}.
\end{proof}

\section{An asymptotic formula for $\gamma(d)$.}

Our aim is to justify the following statement made in the introduction. 

\begin{proposition}
If $\gamma$ is defined as in \eqref{E:gamma},
then
\begin{equation}\label{E:asymp}
\gamma(d)=2^{1/d+O((\log d)^2/d^2)} \quad\text{as~}d\to\infty.
\end{equation}
\end{proposition}

There is no need to suppose that $d$ is an integer here.

\begin{proof}
We begin by deriving an asymptotic formula for $\xi_d$
as $d\to\infty$. 
On the one hand, since 
\[
e^{d\xi_d}\ge\cosh(d\xi_d)=d\cosh(\xi_d)\ge d,
\]
we certainly have $\xi_d\ge (\log d)/d$. 
On the other hand,
since the unimodal function $(\cosh x)/x$ 
takes the same values at $\xi_d$ and $d\xi_d$, 
we must have $\xi_d\le\eta \le d\xi_d$,
where $\eta$ is the point at which 
$(\cosh x)/x$ assumes its minimum. Thus
\[
\frac{e^{d\xi_d}}{2}\le\cosh (d\xi_d)=d\cosh\xi_d\le d\cosh\eta,
\]
whence 
\[
\xi_d= \frac{\log d}{d}+O\Bigl(\frac{1}{d}\Bigr).
\]
This is not yet precise enough. 
Substituting into the equation $\cosh(d\xi_d)=d\cosh(\xi_d)$, 
we obtain
\[
\frac{e^{d\xi_d}}{2}+O\Bigl(\frac{1}{d}\Bigr)\
=d+O\Bigl(\frac{(\log d)^2}{d}\Bigr),
\]
whence
\[
\xi_d=\frac{\log(2d)}{d}+O\Bigl(\frac{(\log d)^2}{d^3}\Bigr).
\]
This is good enough for our needs. 

We now estimate $\gamma(d)$ as $d\to\infty$. 
First of all, we have
\[
d\sinh(\xi_d)=d\xi_d+O(d\xi_d^3)
=\log(2d)+O\Bigl(\frac{(\log d)^3}{d^2}\Bigr).
\]
Also
\[
\sinh(d\xi_d)
=\sinh\Bigl(\log(2d)+O\Bigl(\frac{(\log d)^2}{d^2}\Bigr)\Bigr)
=d+O\Bigl(\frac{(\log d)^2}{d}\Bigr).
\]
Hence
\begin{align*}
\log\gamma(d)
&=\log\Bigl(d\sinh(d\xi_d)+\sinh(d\xi_d)\Bigr)-\frac{d}{d-1}\log d\\
&=\log\Bigl(d+\log(2d)+O\Bigl(\frac{(\log d)^2}{d}\Bigr)\Bigr)-\Bigl(1+\frac{1}{d}
  +O\Bigl(\frac{1}{d^2}\Bigr)\Bigr)\log d\\
&=\log d+\frac{\log(2d)}{d}
  +O\Bigl(\frac{(\log d)^2}{d^2}\Bigr)-\Bigl(\log d+\frac{\log d}{d}
  +O\Bigl(\frac{\log d}{d^2}\Bigr)\Bigr)\\
&=\frac{\log 2}{d}+O\Bigl(\frac{(\log d)^2}{d^2}\Bigr).
\end{align*}
Finally, taking exponentials of both sides, we get \eqref{E:asymp}.
\end{proof}

\begin{acknowledgements}
The first author thanks the organizers of the Conference on Modern Aspects of Complex Geometry, 
held at the University of Cincinnati in honor of Taft Professor David Minda, 
for their kind hospitality and financial support.
\end{acknowledgements}

\bibliographystyle{spmpsci}   
\bibliography{biblist}

\begin{thebibliography}{1}
\providecommand{\url}[1]{{#1}}
\providecommand{\urlprefix}{URL }
\expandafter\ifx\csname urlstyle\endcsname\relax
  \providecommand{\doi}[1]{DOI~\discretionary{}{}{}#1}\else
  \providecommand{\doi}{DOI~\discretionary{}{}{}\begingroup
  \urlstyle{rm}\Url}\fi

\bibitem{LS96}
Lau, E., Schleicher, D.: Symmetries of fractals revisited.
\newblock Math. Intelligencer \textbf{18}(1), 45--51 (1996)

\bibitem{PRR16}
Paris{\'e}, P.O., Ransford, T., Rochon, D.: Tricomplex dynamical systems
  generated by polynomials of odd degree.
\newblock Preprint  (2016)

\bibitem{PR15a}
Paris{\'e}, P.O., Rochon, D.: A study of dynamics of the tricomplex polynomial
  {$\eta^p+c$}.
\newblock Nonlinear Dynam. \textbf{82}(1-2), 157--171 (2015)

\bibitem{PR15b}
Paris{\'e}, P.O., Rochon, D.: Tricomplex dynamical systems generated by
  polynomials of odd degree.
\newblock Preprint  (2015)

\bibitem{Sch04}
Schleicher, D.: On fibers and local connectivity of {M}andelbrot and
  {M}ultibrot sets.
\newblock In: Fractal geometry and applications: a jubilee of {B}eno\^\i t
  {M}andelbrot. {P}art 1, \emph{Proc. Sympos. Pure Math.}, vol.~72, pp.
  477--517. Amer. Math. Soc., Providence, RI (2004)

\end{thebibliography}

\end{document}